\documentclass[11pt]{article}
\usepackage{amsmath}
\usepackage{amsfonts}
\usepackage{bm}
\usepackage{amsmath}
\usepackage{cases}
\usepackage{color}

\newtheorem{thm}{Theorem}
\newtheorem{lemm}{Lemma}[section]
\newtheorem{rem}{Remark}
\newtheorem{prop}{Proposition} 

\newtheorem{rem*}[rem]{{\color {red} Remark}}

\title{Extremal Trudinger-Moser inequality associated with many point vortices with multi-intensities}

\author{
Takashi Suzuki\thanks{Center for Mathematical Modeling and Data Science, Unit of Mathematical Science, Osaka University. Email: suzuki@sigmath.es.osaka-u.ac.jp} and  
Yohei Toyota\thanks{Department of Liberal Studies, National Institute of Technology, Nara College, Email: y-toyota@libe.nara-k.ac.jp} 
}

\date{}

\begin{document}
  \maketitle
\begin{abstract} 
We study the boundedness of extremal 2D Trudinger-Moser inequality associated with many point vortices with multi-intensities. Sufficient conditions for continuously-distributed measures are given in one-sided and two-sided cases.  
\end{abstract}

\section{Introduction}\label{sec1}

Our purpose is to show the boundedness of extremal 2D Trudinger-Moser inequality associated with many point vortices with multi-intensities. 

Let $\Omega$  be a compact Riemann surface without boundary, $P(d\alpha)$ be a probability measure on $I=[-1,1]$ satisfying $1\in \mbox{supp} \ P$, and $\lambda>0$ be a constant. The elliptic equation 
\begin{equation} 
-\Delta v=\lambda\int_I\alpha\left( \frac{e^{\alpha v}}{\int_\Omega e^{\alpha v}}-\frac{1}{\vert \Omega\vert}\right) P(d\alpha), \ \int_\Omega v=0 
 \label{2}
\end{equation}  
arises in Onsager's theory \cite{ons49} as the the mean field limit of the system of many point vortices on $\Omega$. Here, their intensities are distributed as in $P(d\alpha)$, and $\lambda$ stands for the negative inverse temperature (\cite{eyink06, ss08}). 

It is the Euler equation of the variational functional 
\begin{equation} 
J_\lambda(v)=\frac{1}{2}\Vert \nabla v\Vert_2^2-\lambda\int_I\left( \log \int_\Omega e^{\alpha v}\right)P(d\alpha), \ v\in E
 \label{1}
\end{equation} 
for 
\[ E=\{ v\in H^1(\Omega) \mid \int_\Omega v=0\}. \] 
Let  
\begin{equation} 
j_\lambda=\inf_EJ_\lambda. 
 \label{j1}
\end{equation} 
There is an extremal value of $\lambda$ for $j_\lambda>-\infty$ denoted by $\overline{\lambda}$. If the inequality 
\begin{equation} 
j_{\overline{\lambda}}>-\infty, 
 \label{3.5}
\end{equation} 
is valid and this $j_{\overline{\lambda}}$ is not attained, an order structure is realized in this systems as a singular limit of the solution to (\ref{2}). 

If $P=\delta_1$, it holds that $\overline{\lambda}=8\pi$ and also  (\ref{3.5}). 
It is also known that $\overline{\lambda}=16\pi$ if $P=\frac{1}{2}(\delta_1+\delta_{-1})$ as well as inequality (\ref{3.5}) (\cite{os06}).  It holds generally that 
\begin{equation} 
\overline{\lambda}=\inf\left\{ \frac{8\pi P(K_\pm)}{\left(\int_{K_\pm}\vert \alpha\vert P(d\alpha)\right)^2}\mid K_\pm \subset I_\pm \cap \mbox{supp $P$, Borel set}\right\} 
 \label{3}
\end{equation}  
for 
\[ I_+=[0,1], \ I_-=[-1,0]. \] 
Therefore, we have $j_\lambda>-\infty$ for $\lambda<\overline{\lambda}$, while $\lambda>\overline{\lambda}$ implies $j_{\lambda}=-\infty$ (\cite{rs14}). 

Boundedness of $j_{\overline{\lambda}}$ as in (\ref{3.5}) is known if $P(d\alpha)$ is a finite sum of delta-functions (\cite{ors10} Theorem 7.1). 
Here we study continuously distributed $P(d\alpha)$, and show the following theorems, where 
\[ \alpha_\pm =\min\{ I_\pm \cap \mbox{supp} \ (\pm P)\}. \]  
\begin{thm}[one-sided measure]\label{thm2} 
If 
\begin{equation} 
\mbox{supp} \ P\subset I_+, \ \alpha_+>\frac{1}{2}. 
 \label{9.2}
\end{equation} 
it holds that (\ref{3.5}). 
\end{thm}

\begin{thm}[two-sided measure]\label{thm3}
If 
\begin{equation} 
I_\pm\cap \mbox{supp} \ P \neq \emptyset, \ \alpha_\pm>\frac{1}{2}, 
 \label{7x}
\end{equation} 
it holds that (\ref{3.5}). 
\end{thm} 

For the proof of these theorems we use the following facts. 

\begin{prop}\label{pro1}
If $\lambda<\overline{\lambda}$, the value $j_\lambda$ in (\ref{j1}) is attained. 
\end{prop} 
We have, therefore, a sequence $(v_k, \lambda_k)$, $k=1,2,\cdots$, satisfying 
\begin{equation} 
\lambda_k\uparrow \overline{\lambda}, \ v_k\in E, \ J_{\lambda_k}(v_k)=j_{\lambda_k}.  
 \label{3.7}
\end{equation} 

\begin{prop}\label{pro2}
If 
\begin{equation} 
\Vert v_k\Vert_\infty\leq C, 
 \label{3.6}
\end{equation} 
it holds that $j_{\overline{\lambda}}>-\infty$ and this $j_{\overline{\lambda}}$ is attained. 
\end{prop}

In the other case of (\ref{3.6}), that is, 
\begin{equation} 
\Vert v_k\Vert_\infty\rightarrow+\infty. 
 \label{111}
\end{equation} 
we derive a local uniform estimate of Y.Y. Li type \cite{yyli99} up to a subsequence. This estimate follows from the asymptotic behavior of $\{ v_k\}$, called {\it  residual vanishing} and {\it mass identity}. Both of them arise in the context of blowup analysis, and implies   
\begin{equation} 
j_{\lambda_k}=J_{\lambda_k}(v_k)\geq -C. 
 \label{bounded}
\end{equation} 
Inequality (\ref{bounded}) ensures (\ref{3.5}) because the mapping 
\[ \lambda\mapsto j_\lambda=\inf_EJ_\lambda\in [-\infty,+\infty) \] 
is upper semi-continuous.: 
\[ -C\leq \limsup_{k\rightarrow\infty}j_{\lambda_k}\leq j_{\overline{\lambda}}, \quad \limsup_{k\rightarrow\infty}\lambda_k=\overline{\lambda}. \] 

Here we explain the notions of residual vanishing and mass identity in short, under the context of the proof of Theorem \ref{thm2}. Actually in \cite{stz14} it is asserted that the residual vanishing alone implies (\ref{bounded}), which, however, is insufficient, and the other independent property, mass identity, is required to complete the proof. 

To this end we observe that $(v_k, \lambda_k)$, $k=1,2,\cdots$, in (\ref{3.7}) is a solution to (\ref{2}) for $(v, \lambda)=(v_k, \lambda_k)$: 
\begin{equation} 
-\Delta v_k=\lambda_k\int_I\alpha\left( \frac{e^{\alpha v_k}}{\int_\Omega e^{\alpha v_k}}-\frac{1}{\vert \Omega\vert}\right) P(d\alpha), \ \int_\Omega v_k=0.  
 \label{2.1}
\end{equation}  
If (\ref{111}) holds, therefore, blowup analysis of \cite{ors10} is applicable. Up to a subsequence there arises the convergence in ${\cal M}(\Omega\times I)=C(\Omega\times I)'$,   
\begin{equation} 
\lambda_k\frac{e^{\alpha v_k}}{\int_\Omega e^{\alpha v_k}}dxP(d\alpha)\rightharpoonup \left[ \sum_{x_0\in {\cal S}}m(x_0,\alpha)\delta_{x_0}(dx)+r(x,\alpha)dx\right]P(d\alpha), 
 \label{4}
\end{equation}  
where 
\[ 0\leq m(x_0, \cdot)\in L^\infty(I,P), \ 0\leq r\in L^1(\Omega\times I) \] 
and 
\begin{equation} 
{\cal S}={\cal S}_+\cup {\cal S}_-, \ {\cal S}_\pm=\{ x_0\in\Omega \mid \exists x_k\rightarrow x_0, \ \pm v_k(x_k)\rightarrow +\infty\}. 
 \label{5.5}
\end{equation} 
Since the Green function to 
\begin{equation} 
-\Delta v=f-\frac{1}{\vert \Omega\vert}\int_\Omega f, \ \int_\Omega v=0 
 \label{x10}
\end{equation} 
denoted by $G=G(x,x')$ is bounded below (\cite{aubin}),  
we have  
\begin{equation} 
\mbox{supp} \ P\subset I_\pm \ \Rightarrow \ {\cal S}_\mp=\emptyset.  
 \label{sup+}
\end{equation} 
It holds also that $\sharp {\cal S}<+\infty$ and 
\begin{eqnarray} 
& & n^\pm(x_0)\equiv \int_{I_\pm}\vert \alpha\vert m(x_0,\alpha)P(d\alpha)\geq 4\pi, \ x_0 \in {\cal S}_\pm, \ \label{5} \\ 
& & 8\pi\int_Im(x_0,\alpha)P(d\alpha)=\left( \int_I\alpha m(x_0,\alpha)P(d\alpha)\right)^2, \ x_0 \in {\cal S}. \ \label{6}
\end{eqnarray} 
Then the residual vanishing described above means $s=0$, where $s=s^++s^-$ for  
\begin{equation} 
s^\pm(x)= \int_{I_\pm}r(x,\alpha)P(d\alpha). 
 \label{rv}
\end{equation} 

In accordance with the residual vanishing, we notice the following properties.

\begin{prop}\label{pro3}
If (\ref{9.2}) holds, the residual vanishing $s^+=0$ occurs. 
\end{prop} 

\begin{prop}[\cite{sz13}]\label{pro4}
If $\mbox{supp} \ P\subset I_+$ and $s^+=0$ in (\ref{3.7}) with (\ref{111}), we have the following facts: 
\begin{enumerate} 
\item It holds that $\sharp {\cal S}=1$. 
\item The extremal value $\overline{\lambda}$ in (\ref{3}) is attained by $K_+=I_+$. 
\item Putting ${\cal S}=\{x_0\}$, we have  
\[ m(\alpha,x_0)=\overline{\lambda}, \ \mbox{$P(d\alpha)$-a.e.}. \]
\end{enumerate} 
\end{prop} 

Mass identity, on the other hand, is associated with the scaling transformation. We continue to suppose $\mbox{supp} \ P\subset I_+$ and (\ref{111}) to $(v_k, \lambda_k)$, $k=1,2,\cdots$, defined by (\ref{3.7}). There is $x_k\in \Omega$ such that 
\begin{equation} 
\Vert v_k\Vert_\infty=v_k(x_k)\rightarrow +\infty 
 \label{12.2}
\end{equation}  
by (\ref{sup+}). Passing to a subsequence, then, we have $x_0\in \Omega$ such that
\begin{equation} 
x_k\rightarrow x_0\in {\cal S}={\cal S}_+. 
 \label{12.4}
\end{equation}

Here we have the following facts \cite{st18}. First, it holds that  
\begin{equation} 
\sigma_k=e^{-w_k(0)/2}\rightarrow 0 
 \label{12.3}
\end{equation} 
for 
\[ w_k(x)=v_k(x+x_k)-\int_\Omega e^{v_k}. \]
Through the iso-thermal coordinate, second, we introduce the transformation 
\begin{equation} 
x'\in \Omega \ \mapsto \ x=\sigma_k^{-1}(x'-x_k) 
 \label{15.1}
\end{equation} 
in a neighbourhood of $x_k$ in $\Omega$. 

More precisely, we use the local chart 
\[ X':x'\in B_r(x_k)\subset \Omega \ \rightarrow \ {\bf R}^2, \ X'(x_k)=0 \] 
for $0<r\ll 1$ with the induced metric 
\[ g(X')=e^{\xi (X')}(dX_1'^2+dX_2'^2), \ \xi(0)=0. \] 
Then we take a ball with center at the origin in ${\bf R}^2$, denoted by $\tilde B_R\subset X'(B_r(x_k))$ for $0<R\ll 1$.   We put  
\[ X=\sigma_k^{-1}X'\in \tilde B_{\sigma_k^{-1}R},  \ X'\in \tilde B_R \] 
to provide the metric 
\begin{equation}  
\tilde g_k(X)=e^{\xi (X')}(dX_1^2+dX_2^2)  
 \label{pb}
\end{equation}  
on $\tilde B_{\sigma_k^{-1}R}$.  As $k\rightarrow \infty$, this $\tilde g_k$ converges to the standard metric on ${\bf R}^2$, locally uniformly.  
This $X$ is identified with $x$ in (\ref{15.1}) if $\Omega$ is flat, and henceforth, we write (\ref{15.1}) for the transformation $x'\in\Omega\mapsto X\in {\bf R}^2$ in short.

Now we define $\tilde f_k=\tilde f_k(x)$, $k=1,2,\cdots$, by 
\begin{equation} 
\tilde f_k(x)=\sigma_k^2f_k(\sigma_kx+x_k), \ x\in \tilde B_{\sigma_k^{-1}R}, \ 0<R\ll 1 
  \label{scale2}
\end{equation} 
for 
\begin{equation} 
f_k(x)=\lambda_k\int_{I_+}\alpha\frac{e^{\alpha v_k}}{\int_\Omega e^{\alpha v_k}}P(d\alpha). 
 \label{12.1}
\end{equation} 
It holds that (\ref{x10}) for $(v,f)=(v_k, f_k)$. 
A subsequence of $\tilde f_k$, $k=1,2,\cdots$, furthermore, converges to $\tilde f\not \equiv 0$ in $C^2_{loc}({\bf R}^2)$ as is confirmed in $\S$\ref{sec4}.

Under such a situation, mass identity means 
\begin{equation} 
n^+(x_0)=\int_{{\bf R}^2}\tilde f\equiv \beta_0. 
 \label{7.1}
\end{equation} 
Equality (\ref{7.1}) indicates that the total mass of the scaling limit $\tilde f$ of the sequence $f_k$, $k=1,2,\cdots$, is equal to its collapse mass $n^+(x_0)$ in (\ref{5}). By this condition of (\ref{7.1}),  the tail of $w_k$ in the rescaled variable is controllable, which ensures the estimate 
\begin{equation} 
v_k(x)-v_k(x_k)=-\frac{\beta_k}{2\pi}\log\left( 1+\left( \frac{e^{v_k(x_k)}}{\int_\Omega e^{v_k}}\right)^{1/2}\vert x-x_k\vert \right)+O(1) 
 \label{9}
\end{equation} 
uniformly in $x\in B_R(x_0)$, $0<R\ll 1$, with 
\begin{equation} 
\beta_k\leq \beta_0, \ \beta_k\rightarrow \beta_0. 
 \label{14.2}
\end{equation}  

Inequality (\ref{9}) with (\ref{14.2}) is a weak form of the inequality obtained by \cite{yyli99} for $P=\delta_1$, and the above argument of controlling the tail of rescaled $v_k$ is due to \cite{lin07} .

Mass identity again follows from (\ref{9.2}), and then, (\ref{bounded}) is obtained by (\ref{9})-(\ref{14.2}), which concludes the proof of Theorem \ref{thm2}. 

As for Theorem \ref{thm3}, we reduce it to the case of one-sided measure treated in Theorem \ref{thm2}, using (\ref{6}).  It holds that ${\cal S}_+\cap {\cal S}_-=\emptyset$ and $\sharp {\cal S}_\pm \leq 1$ by (\ref{7x}), and then (\ref{3.5}) follows from the method of \cite{os06} developed for the case of $P=\frac{1}{2}(\delta_1+\delta_{-1})$,  where (\ref{2}) takes the form of 
\begin{equation} 
-\Delta v=\frac{\lambda}{2}\left( \frac{e^{v}}{\int_\Omega e^v}-\frac{e^{-v}}{\int_\Omega e^{-v}}\right), \ \int_\Omega v=0. 
 \label{neutral1}
\end{equation} 

Here we provide several remarks on the condition (\ref{9.2}). First, Theorem 4 of \cite{st19} says that the residual vanishing, the mass identity, and $P(\{1\})>0$ imply 
\[ 
\frac{8\pi}{\left( \int_{I_+}\alpha P(d\alpha)\right)^2}
\leq 8\pi P(\{1\})^{-1}. 
\] 
In fact, we have 
\begin{equation} 
P(\{1\})<\left( \int_{I_+}\alpha P(d\alpha)\right)^2 
 \label{9.1}
\end{equation} 
if (\ref{9.2}) and $P(\{1\})>0$ hold. To see this property, we put 
\[ P=\tau\delta_1+(1-\tau)P_0, \ 0<\tau<1, \] 
where $P_0=P_0(d\alpha)$ is a probability measure on $(\frac{1}{2},1)$. Then it holds that 
\begin{eqnarray*} 
 \left( \int_{I_+}\alpha P(d\alpha)\right)^2 & = & \left( \tau+(1-\tau)\int_{I_+}\alpha P(d\alpha)\right)^2 \\ 
& \geq & \left( 
\tau+\frac{1}{2}(1-\tau)\right)^2 =\frac{1}{4}(1+\tau)^2>\tau=P(\{1\}) 
\end{eqnarray*} 
and hence (\ref{9.1}). 

Second, the condition (\ref{9.2}) is almost optimal for residual vanishing. To see this property, we take 
\begin{equation} 
P=\tau\delta_1+(1-\tau)\delta_\gamma, \ 0<\tau<1, \ 0<\gamma<1. 
 \label{9.3}
\end{equation} 
The the residual vanishing implies  
\[ 
\frac{8\pi}{\left(\int_{I_+}\alpha P(d\alpha)\right)^2}
\leq \frac{8\pi P(K_+)}{\left( \int_{K_+}\alpha P(d\alpha)\right)^2} 
\] 
for $K=\{1\}$ by Proposition \ref{pro4}, which means  
\begin{equation} 
\gamma\geq \frac{\sqrt{\tau}}{1+\sqrt{\tau}}. 
 \label{tau}
\end{equation} 
Here we note that the right-hand side converges to $\frac{1}{2}$ as $\tau\uparrow 1$. 

Third, (\ref{9.2}) is optimal also for mass identity. In fact, Theorem 1.2 of \cite{st18} says that for $P(d\alpha)$ in the form of (\ref{9.2}), mass identity holds and fails for 
\[ \gamma>\frac{\sqrt{\tau}}{1+\sqrt{\tau}} \quad \mbox{and} \quad \gamma<\frac{\sqrt{\tau}}{1+\sqrt{\tau}},  \]
respectively. 

In spite of these optimalities for residual vanishing and mass identity, the condition $\alpha_+>\frac{1}{2}$ of (\ref{9.2}) is not essential for (\ref{3.5}) to hold. Similarly to the case of discrete $P(d\alpha)$, there may arise this boundedness for any continuously distributed $P$. 

We conclude this section with several known facts concerning the other functionals. The functional $J_\lambda$ in (\ref{1}) is associated with deterministic intensities, and for the moment we put 
\[ J_\lambda=J_\lambda^d, \quad j_\lambda=j_\lambda^d, \quad \overline{\lambda}=\overline{\lambda}_d, \] 
to distinguish it with the other functionals. 

First, if the intensities of point vortices are independent random variables, subject to the same distribution $P(d\alpha)$ on $I=[-1,1]$, their mean field equation is described by    
\begin{equation} 
-\Delta v=\frac{\lambda}{\iint_{\Omega\times I}e^{\alpha v}}\int_I\alpha\left( e^{\alpha v} -\frac{1}{\vert \Omega\vert}\int_\Omega e^{\alpha v}\right)P(d\alpha), \ \int_\Omega v=0 
 \label{x11}
\end{equation} 
(\cite{ner04}). We thus obtain 
\[ -\Delta v=\frac{\lambda}{\int_\Omega e^v+e^{-v} \ dx}\left( e^v-e^{-v}\right), \ \int_\Omega v=0 \]
for $P=\frac{1}{2}(\delta_1+\delta_{-1})$, differently from (\ref{neutral1}). The associated variational functional with (\ref{x11}) is 
\[ J^s_\lambda(v)=\frac{1}{2}\Vert \nabla v\Vert_2^2-\lambda\log\left( \iint_{\Omega\times I}e^{\alpha v}\right)P(d\alpha)dx, \ v\in E. \] 
Letting  
\[ j_\lambda^s=\inf_EJ^s_\lambda, \]
we obtain 
\[ \sup\{ \lambda \mid j^s_\lambda>-\infty\}=8\pi, \] 
provided that $1\in \mbox{supp} \ P$, and it holds that $j^s_{8\pi}>-\infty$ 
(\cite{rz12, sz13}). 

Analogous results hold when $\Omega\subset {\bf R}^2$ is a bounded domain with smooth boundary $\partial\Omega$. For the variational functionals 
\[ J_\lambda^{d0}(v)=\frac{1}{2}\Vert \nabla v\Vert_2^2-\lambda\int_I\left(\log \int_\Omega e^{\alpha v}dx\right) P(d\alpha), \ v\in H^1_0(\Omega) \] 
and 
\[ J_\lambda^{s0}(v)=\frac{1}{2}\Vert \nabla v\Vert_2^2-\lambda\log\left( \iint_{\Omega\times I}e^{\alpha v}dxP(d\alpha)\right), \ v\in H^1_0(\Omega), \] 
the extremal values for $j^{i0}_\lambda>-\infty$, denoted by $\overline{\lambda}_{i0}$, $i=d,s$, are given by $\overline{\lambda}_{d0}=\overline{\lambda}_d$ and $\overline{\lambda}_{s0}=8\pi$ (if $1\in \mbox{supp} \ P$), respectively, where 
\[ j_\lambda^{i0}=\inf_EJ_\lambda^{i0}. \] 
It holds also that $j_{8\pi}^{s0}>-\infty$. In fact, $\overline{\lambda}_{s0}=8\pi$ and $j^{s0}_{8\pi}>-\infty$ follow from the other Trudinger-Moser inequality 
\[ \int_\Omega e^{4\pi v^2}\leq C, \quad v\in H^1_0(\Omega), \ \Vert \nabla v\Vert_2=1 \] 
as in \cite{rz12}. 

As for $\overline{\lambda}_{d0}$, first, we take discrete $P(d\alpha)$. In this case the result \cite{sw05b} on logarithmic HLS inequality is applicable. By the duality argument we  have $j_{\lambda}^{d0}>-\infty$ and $j_\lambda^{d0}=-\infty$ for $\lambda<\overline{\lambda}_d$ and $\lambda>\overline{\lambda}_d$, respectively. It holds also that $j_{\overline{\lambda}_d}>-\infty$. 

Second, even if $P(d\alpha)$ is continuously distributed,  we obtain $j_\lambda>-\infty$ for $\lambda<\overline{\lambda}_d$, approximating it by discrete measures  as in \cite{rs14}. Then, analogous result to Theorem \ref{thm2} arises concerning $j_{\overline{\lambda}_d}>-\infty$, becasuse the exclusion of boundary blowup arises for the sequence of solutions to the associated Euler equation 
\[ -\Delta v=\lambda\int_I\alpha \frac{e^{\alpha v}}{\int_\Omega e^{\alpha v}} \ P(d\alpha), \ \left. v\right\vert_{\partial\Omega}=0. \] 
In fact, in this case of one-sided $P(d\alpha)$ the method of moving planes works similarly to the case of the single intensity \cite{ns90},  
see \cite{suz22} for example. It seems, however, that this exclusion of boundary blowup is not known if $P(d\alpha)$ is two-sided, in spite of \cite{rt16} concerning the Euler equation to $J_\lambda^{s0}$, 
\begin{equation} 
-\Delta v=\lambda\int_I\alpha \frac{e^{\alpha v}}{\iint_{\Omega\times I} e^{\alpha v}}P(d\alpha), \ \left. v\right\vert_{\partial\Omega}=0. 
 \label{se}
\end{equation}

This paper is composed of four sections. Propositions \ref{pro1}-\ref{pro4}, Theorem \ref{thm2}, and Theorem \ref{thm3} are proven in Section \ref{sec2}, Section \ref{sec4}, and Section \ref{sec5}, respectively.  

\section{Preliminaries}\label{sec2}

We begin with the following proof. 

\vspace{2mm} 

{\it Proof of Proposition \ref{pro1}:} \ Given $\lambda<\overline{\lambda}$, we have $v_k\in E$, $k=1,2,\cdots$, such that 
\[ \lim_{k\rightarrow \infty}J_\lambda(v_k)=j_\lambda>-\infty. \] 
For $\lambda<\lambda_1<\overline{\lambda}$, it holds that  
\begin{eqnarray*} 
j_\lambda & = & \frac{1}{2}\Vert \nabla v_k\Vert_2^2-\lambda\int_I\left( \log \int_\Omega e^{v_k}\right)P(d\alpha) \\ 
& = & \frac{\lambda}{\lambda_1}\left\{ J_{\lambda_1}(v_k)+\frac{1}{2}\left(1-\frac{\lambda}{\lambda_1}\right)\Vert \nabla v_k\Vert_2^2\right\} \\ 
& \geq & \frac{\lambda}{\lambda_1}j_{\lambda_1}+\frac{1}{2}\left(1-\frac{\lambda}{\lambda_1}\right)\Vert \nabla v_k\Vert_2^2 
\end{eqnarray*} 
and  hence 
\begin{equation} 
\Vert \nabla v_k\Vert_2\leq C. 
 \label{vbound}
\end{equation}  
Thus a subsequence takes the weak convergence 
\[ v_k\rightharpoonup v \ \mbox{in $E$}, \] 
and  hence 
\[ \liminf_{k\rightarrow\infty}\Vert \nabla v_k\Vert_2^2\geq \Vert \nabla v\Vert_2^2. \]  
Then it follows that  
\[ \lim_{k\rightarrow\infty}\int_\Omega e^{\alpha v_k}=\int_\Omega e^{\alpha v}, \ \forall \alpha \in I \]
from Trudinger-Moser-Fontana's inequality \cite{fon93}, 
\begin{equation} 
\int_\omega e^{4\pi v^2}\leq C, \quad v\in E, \ \Vert \nabla v\Vert_2\leq 1, 
 \label{fontana}
\end{equation} 
see \cite{st98} for example. It follows also that 
\[ \vert \Omega \vert \leq \int_\Omega e^{\alpha v_k}\leq \int_\Omega e^{\vert v_k\vert}\leq C \] 
from (\ref{vbound})-(\ref{fontana}), and therefore,  
\[ \lim_{k\rightarrow\infty}\int_I\left( \log \int_\Omega e^{\alpha v_k}\right)P(d\alpha)=\int_I\left( \log \int_\Omega e^{\alpha v}\right)P(d\alpha) \] 
by the dominated convergence theorem. We thus obtain 
\[ j_\lambda=\lim_{k\rightarrow \infty}J_\lambda(v_k)\geq J_\lambda(v)\geq j_\lambda \] 
and hence $j_\lambda=J_\lambda(v)$. \hfill $\Box$

\vspace{2mm} 
 
Now we show the following proof.  

\vspace{2mm} 

{\it Proof of Proposition \ref{pro2}:} \ The result follows from the standard elliptic regularity. In fact we have  
\begin{equation} 
-\Delta v_k=f_k-\frac{1}{\vert \Omega\vert}\int_\Omega f_k, \ \int_\Omega v_k=0 
 \label{18.1}
\end{equation}  
for $f_k=f_k(x)$ defind by (\ref{12.1}). Then, by  
\[ \Vert f_k\Vert_\infty\leq \frac{\lambda_k}{\vert \Omega\vert}e^{\Vert v_k\Vert_\infty}\leq C \] 
the $L^p$ estimate and Morrey's inequality ensure 
\[ \Vert v_k\Vert_{C^{1,\theta}}\leq C, \ 0<\theta<1 \] 
and hence $\Vert f_k\Vert_{C^{1,\theta}}\leq C$. Then we obtain 
\[ \Vert v_k\Vert_{C^{2,\theta}}\leq C \] 
by the Schauder estimate. A subsequence satisfies $v_k\rightarrow v$ in $C^2$, and therefore, 
\[ \lim_{k\rightarrow \infty}J_{\lambda_k}(v_k)=J_{\overline{\lambda}}(v)=j_{\overline{\lambda}}.  \] 
\hfill $\Box$

\vspace{2mm} 

We turn to the following proof. 

\vspace{2mm} 

{\it Proof of Proposition \ref{pro3}:} \ We assume $\mbox{supp} \ P \subset I_+$. By (\ref{sup+}) it holds that 
\begin{eqnarray*} 
m^+(x_0) & \equiv& \int_{I_+}m(x_0, \alpha)P(d\alpha) \\ 
& \geq & n^+(x_0)=\int_{I_+}\alpha m(x_0, \alpha)P(d\alpha)\geq 4\pi 
\end{eqnarray*} 
for $x_0\in {\cal S}={\cal S}_+$. Then equality (\ref{6}) implies 
\[ 8\pi m^+(x_0)=\left( \int_{I_+}\alpha m(x_0, \alpha)P(d\alpha)\right)^2\leq m^+(x_0)^2 \] 
andtherefore, 
\[ m_+(x_0)\geq 8\pi. \] 
It thus holds that 
\[ n^+(x_0)^2=\left( \int_{I_+}\alpha m(\alpha, x_0)P(d\alpha)\right)^2=8\pi m^+(x_0)\geq (8\pi)^2 \]
or, 
\begin{equation} 
n^+(x_0)\geq 8\pi, \ x_0\in {\cal S}_+={\cal S}. 
 \label{13.9}
\end{equation} 

Here we obtain 
\begin{equation} 
f_k\rightharpoonup n^+(x_0)\delta_{x_0}+\tilde s^+(x) \ \mbox{in ${\cal D}'(B_R(x_0))$}, \quad 0<R\ll 1  
 \label{14} 
\end{equation} 
for 
\[ \tilde s^\pm(x)=\int_{I_\pm}\vert \alpha\vert r(x,\alpha)P(d\alpha)\geq 0 \] 
by (\ref{4}). It holds also that  
\[ v_k\rightarrow \exists v \ \mbox{in $W^{1.q}(B_R(x_0))$}, \ 1\leq q<2 \] 
by \cite{bs73}, and locally uniformly in $B_R(x_0)\setminus \{ x_0\}$. Then the standard property of the Green function \cite{aubin} ensures 
\begin{eqnarray*} 
v(x) & \geq & \frac{n^+(x_0)}{2\pi}\log \frac{1}{d(x,x_0)}-C \\ 
& \geq & 4\log \frac{1}{d(x,x_0)}-C, \quad x\in B_R(x_0)\setminus \{x_0\}. 
\end{eqnarray*}  
where $d$ denotes the geodesic distance on $\Omega$. 

>From $\alpha_+>\frac{1}{2}$ it follows that 
\[ +\infty=\int_\Omega e^{\alpha_+v}\leq \liminf_{k\rightarrow\infty}\int_\Omega e^{\alpha_+v_k}, \] 
and therefore, 
\[ \mbox{supp} \ \mu(dxd\alpha)\subset {\cal S}_+\times I_+ \] 
by (\ref{4}), where  
\begin{equation} 
\mu(dxd\alpha)=\left[ \sum_{x_0\in {\cal S}}m(x_0, \alpha)\delta_{x_0}(dx)+r(x, \alpha)dx\right]P(d\alpha). 
 \label{mu}
\end{equation} 
Then the residual vanishing $s^+=0$ follows. \hfill $\Box$ 

\vspace{2mm} 

{\it Proof of Proposition \ref{pro4}:} \ From the assumption it follows that 
\begin{eqnarray} 
& & \mu_k(dxd\alpha)\equiv \lambda_k\frac{e^{\alpha v_k}}{\int_\Omega e^{\alpha v_k}}dxP(d\alpha) \label{muk} \\ 
& & \quad \rightharpoonup \mu(dxd\alpha)=\sum_{x_0\in {\cal S}}m(x_0, \alpha)\delta_{x_0}(dx)P(d\alpha), \nonumber 
\end{eqnarray}  
which implies
\[ \iint_{I_+\times \Omega}\mu_k=\lambda_k \ \rightarrow \ \langle 1, \mu(dxd\alpha)\rangle=\sum_{x_0\in {\cal S}}\int_{I_+}m(x_0, \alpha)P(d\alpha) \] 
and 
\[ 
\iint_{I_+\times \Omega}\alpha \mu_k=\lambda_k\int_{I_+}\alpha P(d\alpha) \ \rightarrow \ 
\langle \alpha, \mu(dxd\alpha)\rangle=\sum_{x_0\in {\cal S}}\int_\Omega \alpha m(x_0,\alpha)P(d\alpha). 
\]  
Hence we obtain 
\[ \overline{\lambda}=\sum_{x_0\in {\cal S}}\int_{I_+}m(x_0, \alpha)P(d\alpha) \] 
and 
\[ \overline{\lambda}\int_{I_+}\alpha P(d\alpha)=\sum_{x_0\in {\cal S}}\int_{I_+} \alpha m(x_0,\alpha)P(d\alpha). \]  

We have, on the other hand, (\ref{6}), and therefore,  
\begin{eqnarray} 
8\pi\overline{\lambda} & = & \sum_{x_0\in {\cal S}}\left( \int_{I_+}\alpha m(x_0,\alpha)P(d\alpha)\right)^2 \nonumber\\ 
& \leq & \left(\sum_{x_0\in {\cal S}}\alpha m(x_0,\alpha)P(d\alpha)\right)^2=\overline{\lambda}^2\left( \int_{I_+}\alpha P(d\alpha)\right)^2 
 \label{14.1}
\end{eqnarray} 
by $\mbox{supp} \ P\subset I_+$. We thus obtain 
\[ \frac{8\pi}{\left( \int_{I_+}\alpha P(d\alpha)\right)^2}\leq \overline{\lambda}, \] 
and therefore, $\overline{\lambda}$ in (\ref{3}) is attained by $K_+=I_+$. The equality holds in (\ref{14.1}). and hence $\sharp \ {\cal S}=1$. 

Finally, it holds that 
\[ \int_\Omega \mu_kdx=\lambda_kP(d\alpha)\rightharpoonup \overline{\lambda}P(d\alpha)=\int_\Omega \mu(dxd\alpha)=m(x_0,\alpha)P(d\alpha) \] 
in ${\cal M}(I)$ for ${\cal S}=\{x_0\}$, and therefore, 
\[ m(x_0, \alpha)=\overline{\lambda}, \ \mbox{$P(d\alpha)$ a.e.}.  \] 
\hfill $\Box$ 

\vspace{2mm} 

The minimum $\overline{\lambda}$ of (\ref{3}) at $K_+=I_+$ in Proposition \ref{pro4} is strict if $\alpha_+>\frac{1}{2}$. This property is used to establish mass identity in $\S$\ref{sec4}. 

\begin{prop}\label{pro5}
Under (\ref{9.2}) it holds that 
\[ \overline{\lambda}<\frac{8\pi P(K_+)}{\left( \int_{K_+}\alpha P(d\alpha)\right)^2}, \] 
where 
\[ K_+=(\alpha_\ast,1], \ K_+=[\alpha_\ast,1] \ \ \mbox{for $\alpha_\ast>\alpha_+$} \] 
or  
\[ K_+=(\alpha_+,1] \ \ \mbox{if $P(\{\alpha_+\})>0$}. \]  
\end{prop} 

{\it Proof:} \ Putting 
\[ a\equiv \int_{K_+}\alpha P(d\alpha)\leq b\equiv \int_{K_+}P(d\alpha) \] 
and 
\[ t=\int_{L_+}P(d\alpha)>0, \ L_+=I_+\setminus K_+,  \] 
we obtain 
\[ \int_{I_+}\alpha P(d\alpha)=a+\int_{L_+}\alpha P(d\alpha)\geq a+\alpha_+t. \] 
Since 
\[ \frac{1}{\left( \int_{I_+}\alpha P(d\alpha)\right)^2}\leq \frac{b+t}{(a+\alpha_+t)^2}, \quad \frac{P(K_+)}{\left( \int_{K_+}\alpha P(d\alpha)\right)^2}=\frac{b}{a^2}, \] 
it suffices to confirm  
\begin{equation} 
\frac{b+t}{(a+\alpha_+t)^2}<\frac{b}{a^2}, \ t>0. 
 \label{claim}
\end{equation} 

In fact, putting 
\[ g(t)=\log \frac{b+t}{(a+\gamma t)^2}=\log(b+t)-2\log(a+\gamma t),  \ \gamma=\alpha_+>\frac{1}{2} \] 
we obtain 
\begin{eqnarray*} 
g'(t) & = & \frac{1}{b+t}-\frac{2\gamma}{a+t\gamma} = \frac{1}{(b+t)(a+\gamma t)}\{ a+t\gamma-2\gamma (b+t)\} \\ 
& & \leq \frac{1}{(b+t)(a+\gamma t)}(b+t)(1-2\gamma)<0 
\end{eqnarray*}  
by $\gamma>\frac{1}{2}$, and therefore, (\ref{claim}). \hfill $\Box$ 

\vspace{2mm}

Concluding this section, we note the following fact. 

\begin{rem} 
If 
\[ P=\sum_{i=1}^n\tau_i\delta_{\gamma_i}, \ n\geq 2 \] 
with 
\[ 0<\gamma_1<\cdots<\gamma_n=1, \ \tau_i>0, \ \sum_{i=1}^n\tau_i=1, \ 0<\gamma_{n-1}\ll 1, \] 
the non-residual vanishing $s^+\not \equiv 0$ arises in (\ref{4})-(\ref{6}) for $(v_k, \lambda_k)$, $k=1,2, \cdots$, satisfying (\ref{3.7}) and (\ref{111}).  
\end{rem} 

{\it Proof:} \ If residual vanishing occurs, we have 
\[ \frac{8\pi}{\left( \int_{I_+}\alpha P(d\alpha\right)^2}\leq \frac{8\pi P(K_+)}{\left( \int_{K_+}\alpha P(d\alpha)\right)^2} \] 
or, 
\begin{equation} 
\left( \int_{K_+}\alpha P(d\alpha)\right)^2\leq P(K_+)\cdot \left( \int_{I_+}\alpha P(d\alpha)\right)^2 
 \label{14.9}
\end{equation} 
for any $K_+\subset\{\gamma_i \mid 1\leq i\leq n\}$ by Proposition \ref{pro4}. For $K=\{\gamma_n\}$ inequality (\ref{14.9}) means 
\[ \tau_n\leq \left( \sum_{i=1}^{n-1}\tau_i \gamma_i+\tau_n\right)^2 \] 
which is impossible for $0<\gamma_{n-1}\ll 1$ by $0<\tau_n<1$. \hfill $\Box$ 

\vspace{2mm} 

As is noted in $\S$\ref{sec1}, the residual vanishing implies (\ref{tau}) for $P(d\alpha)$ in (\ref{9.3}) and hence non-residual vanishing occurs iff $\gamma<\frac{\sqrt{\tau}}{1+\sqrt{\tau}}$ in this case. So far, the non-residual vanishing is confirmed for $P=\delta_1$ if the source terms arise and two of them make a  collision (\cite{lt16}).

\section{One-sided distibutions: proof of Theorem \ref{thm2}}\label{sec4}

In this section we take the case (\ref{9.2}). Let $(v_k, \lambda_k)$, $k=1,2,\cdots$, be as in (\ref{3.7}), and assume (\ref{12.2})-(\ref{12.4}). 
Then we apply the scaling transformation of \cite{st18} to derive (\ref{9})-(\ref{14.2}). In the next section we confirm more details of the estimates used here in a slight more generality. 

 We put 
 \begin{equation} 
w_k^\alpha(x)=\alpha v_k(x+x_k)-\log\int_\Omega e^{\alpha v_k}, \ w_k=w_k^1 
 \label{wk}
\end{equation} 
for $0<\alpha\leq 1$.  It holds that 
\[ \frac{d}{d\alpha}w_k^\alpha(0)=v_k(x_k)-\frac{\int_\Omega v_ke^{\alpha v_k}}{\int_\Omega e^{\alpha v_k}}\geq 0 \] 
by $v_k(x_k)=\Vert v_k\Vert_\infty$, and hence 
\[ w_k(0)\geq w_k^\alpha(0). \] 
It holds also that 
\begin{eqnarray*} 
w_k^\alpha(0) & = & \alpha v_k(x_k)-\log \int_\Omega e^{\alpha v_k} = \alpha\left( v_k(x_k)+\log \int_\Omega e^{v_k} \right)-\log\int_\Omega e^{\alpha v_k} \\ 
& = & \alpha w_k(0)+\log\frac{\left( \int_\Omega e^{v_k}\right)^\alpha}{\int_\Omega e^{\alpha v_k}} 
\end{eqnarray*} 
with 
\[ \frac{\left( \int_\Omega e^{v_k}\right)^\alpha}{\int_\Omega e^{\alpha v_k}}=\vert \Omega \vert^{\alpha-1}\cdot \frac{\left( \frac{1}{\vert \Omega\vert}\int_\Omega e^{v_k}\right)^\alpha}{\frac{1}{\vert \Omega\vert}\int_\Omega e^{\alpha v_k}}\geq \vert \Omega\vert^{\alpha-1} \] 
and therefore, 
\begin{equation} 
w_k(0)\rightarrow+\infty \ \ \Leftrightarrow \ \ w_k^\alpha(0)\rightarrow+\infty, \ 0<\alpha\leq 1. 
 \label{xx12}
\end{equation} 

Here we have 
\[ w_k^\alpha(x)=\log \frac{e^{\alpha v_k(x+x_k)}}{\int_\Omega e^{\alpha v_k}}\leq w_k^\alpha(0) \leq w_k(0), \] 
and therefore, $w_k(0)\leq C$ implies $\Vert f_k\Vert_\infty\leq C$ in (\ref{12.1}):   
\[ f_k(x)=\lambda_k\int_{I_+}\alpha\frac{e^{\alpha v_k(x)}}{\int_\Omega e^{\alpha v_k}}P(d\alpha). \]
Then we obtain $\Vert v_k\Vert_\infty\leq C$ by (\ref{18.1}), a contradiction. It thus holds that (\ref{12.3}), passing to a subsequence, and also 
\[ w_k^\alpha(0)\rightarrow +\infty, \ 0<\alpha \leq 1 \]
by (\ref{xx12}).  

Recall the transformation (\ref{15.1}) 
induced by the iso-thermal coordinate on $\Omega$, and put 
\begin{equation} 
\tilde w_k^\alpha(x)=w_k^\alpha(\sigma_kx)+2\log\sigma_k, \ \tilde w_k=\tilde w_k^1, \quad  x\in \tilde B_{\sigma_k^{-1}R}.  
 \label{scale1}
\end{equation}  
Let $\tilde g=\tilde g_k$ be the metric defined by (\ref{pb}). Then it follows that 
\begin{eqnarray*} 
\Delta_{\tilde g}\tilde w_k & = & e^{\xi_k(X')}\Delta_{X}w_k(\sigma_k X)  \\ 
& = & \sigma_k^2e^{\xi_k(X')}(\Delta_{X'}w_k)(\sigma_kX) \\ 
& = & \sigma_k^2\Delta_gw_k, 
\end{eqnarray*}  
and therefore, 
\begin{eqnarray*} 
-\Delta \tilde w_k(x) & = & -\sigma_k^2\Delta v_k(\sigma_kx) \\ 
&  = & \sigma_k^2\lambda_k\int_{I_+}\alpha\left( e^{\alpha v_k(\sigma_kx)-\log\int_\Omega e^{\alpha v_k}}-\frac{1}{\vert \Omega\vert}\right) \ P(d\alpha) \\ 
& = &  \lambda_k\int_{I_+}\alpha\left( e^{\tilde w_k^\alpha(x)}-\frac{\sigma_k^2}{\vert \Omega\vert}\right) \ P(d\alpha). 
\end{eqnarray*} 

For $\tilde f_k(x)$ defined by (\ref{scale2}), we obtain  
\begin{equation} 
\tilde f_k=\lambda_k\int_{I_+}\alpha e^{\tilde w_k^\alpha}P(d\alpha),  
 \label{scale3}
\end{equation} 
and therefore,   
\[ -\Delta \tilde w_k=\tilde f_k-\lambda_k\frac{\sigma_k^2}{\vert \Omega\vert}\int_{I_+}\alpha P(d\alpha) \ \mbox{in $\tilde B_{\sigma_k^{-1}R}$}. \]  
It holds also that 
\[ \tilde w_k^\alpha\leq \tilde w_k^\alpha(0)\leq \tilde w_k(0)=0. \] 
From (\ref{4}) and (\ref{12.1}), furthermore, it follows that 
\begin{equation} 
f_kdx \rightharpoonup \nu^+(dx)\equiv s^+dx+\sum_{x_0\in {\cal S}_+}n^+(x_0)\delta_{x_0}(dx)
 \label{7}
\end{equation} 
in ${\cal M}(\Omega)=C(\Omega)'$. 

Similarly to Proposition 2.1 of \cite{st18}, then, a subsequence admits the convergence
\begin{equation} 
\tilde w_k\rightarrow \tilde w, \ \tilde f_k\rightarrow \tilde f \quad \mbox{in $C^2_{loc}({\bf R}^2)$}.  
 \label{conv}
\end{equation} 
There arises also that 
\begin{equation} 
-\Delta \tilde w=\tilde f\not\equiv 0, \ \tilde w\leq \tilde w(0)=0, \ 0\leq \tilde f \leq \overline{\lambda}\int_{I_+}\alpha P(d\alpha) \quad \mbox{in ${\bf R}^2$}  
 \label{eq1}
\end{equation}  
and 
\begin{equation} 
\int_{{\bf R}^2}e^{\tilde w}\leq 1, \ \beta_0\equiv \int_{{\bf R}^2}\tilde f\leq \overline{\lambda}\int_{I_+}\alpha P(d\alpha).  
 \label{eq2}
\end{equation}
 
Here we note the following fact. 
\begin{prop}\label{pro8} 
If $\mbox{supp} \ P\subset I_+$ the equality 
\begin{equation} 
\beta_0\equiv \int_{{\bf R}^2}\tilde f=\overline{\lambda}\int_{I_+}\alpha P(d\alpha). 
 \label{19}
\end{equation} 
is equivalent to the residual vanishing $s^+=0$ in (\ref{rv}) and mass identity (\ref{7.1}) for $x_0$ defined by (\ref{12.2})-(\ref{12.4}).  
\end{prop} 
{\it Proof:}  \ First, we have 
\begin{eqnarray} 
& & \langle 1, f_kdx\rangle=\lambda_k\int_{I_+}\alpha P(d\alpha) \nonumber\\ 
& & \quad \rightarrow \overline{\lambda}
\int_{I_+}\alpha P(d\alpha)
=\sum_{x_0'\in {\cal S}}n^+(x_0')+\Vert s^+\Vert_1\geq n^+(x_0) 
 \label{42.1}
\end{eqnarray} 
by (\ref{7}) and  (\ref{12.1}).  For $r<\sigma_k^{-1}R$, on the other hand, we have 
\[ 
\int_{\tilde B_r}\tilde f\leq 
\int_{\tilde B_{\sigma_k^{-1}R}}\tilde f_k=\int_{B_R(x_k)}f_k. 
\]  
Letting $k\rightarrow\infty$,  $R\rightarrow 0$, and $r\rightarrow+\infty$ in turn, we obtain 
\begin{equation}  
\beta_0=\int_{{\bf R}^2}\tilde f\leq n^+(x_0) 
 \label{42.2}
\end{equation} 
by (\ref{7}).  If (\ref{19}) holds, therefore, there arises ${\cal S}=\{x_0\}$, (\ref{7.1}), and $s^+=0$. 

Suppose, second, $s^+=0$ in (\ref{rv}) and (\ref{7.1}) for $x_0$ defined by (\ref{12.4}).  Then we obtain ${\cal S}=\{x_0\}$ with $s^+=0$, and therefore, 
\begin{equation} 
n^+(x_0)=\overline{\lambda}\int_{I_+}\alpha P(d\alpha) 
 \label{42.3}
\end{equation} 
by (\ref{42.1}).  Then (\ref{7.1}) and (\ref{42.3}) imply (\ref{19}). \hfill $\Box$ 

\begin{rem} 
As in Lemma 3.2 of \cite{st18} we have  
\begin{equation} 
\lim_{r\rightarrow\infty}\lim_{k\rightarrow\infty}\int_{\tilde B_{\sigma_k^{-1}R}\setminus \tilde B_r}\tilde f_k=0 
 \label{rr}
\end{equation} 
for any $R\gg 1$ under (\ref{rv}) and (\ref{7.1}).  
\end{rem} 

To confirm (\ref{19}), we use the following fact proven by \cite{st18}, where 
\[ \alpha_{\inf}=\inf_{\cal A}\alpha \] 
for 
\[ {\cal A}=\{ \alpha \in \mbox{supp} \ P \mid \limsup_{k\rightarrow\infty}(w_k^\alpha(0)-w_k(0))>-\infty\}. \] 
It is obvious that $1\in {\cal A}\neq\emptyset$.  

\begin{prop}\label{pro9} 
Assume $\mbox{supp} \ P\subset I_+$.  
\begin{enumerate} 
\item There is $\tilde \psi \in L^1(I_+, P(d\alpha))$ such that $0\leq \tilde\psi\leq 1$, $P$-a.e., and  
\[ \beta_0=\overline{\lambda}\int_{I_+}\alpha \tilde \psi(\alpha)P(d\alpha). \] 
\item There is $\alpha_d\in I_+$ such that 
\[ \tilde \psi=\chi_{I_d}, \ \mbox{$P$-a.e.}, \] 
where $I_d=[\alpha_d,1]$ or $I_d=(\alpha_d, 1]$. 
\item If $\alpha_d\leq \alpha_{\inf}$ it holds that 
\begin{equation} 
\tilde \psi=\chi_{I_{\inf}}, \ \mbox{$P$-a.e.}. 
 \label{47.7}
\end{equation} 
where 
\[ I_{\inf}=\left\{ \begin{array}{ll} 
[\alpha_{\inf}, 1], & \alpha_{\inf}\in {\cal A} \\ 
(\alpha_{\inf}, 1], & \alpha_{\inf}\not \in {\cal A}. \end{array} \right. \]  
\item It holds that 
\[ 8\pi\int_{I_+}\tilde \psi(\alpha)P(d\alpha)=\overline{\lambda}\left( \int_{I_+}\alpha\tilde \psi(\alpha)P(d\alpha)\right)^2. \] 
\end{enumerate}  
\end{prop} 

Now we show the following facts. 

\begin{lemm}\label{lem4.1}
If (\ref{9.2}) holds with $P(\{ \alpha_+\})>0$, there arises that 
\[ \alpha_+=\alpha_{\inf}\in {\cal A}. \] 
\end{lemm} 

{\it Proof:} \ It is obvious that $\alpha_+\leq \alpha_{\inf}$. We assume 
\begin{equation} 
\alpha_+<\alpha_{\inf} 
 \label{42.0}
\end{equation} 
to derive a contradiction. 

First, if $\alpha_d>\alpha_{\inf}$, the second and the fourth cases of Proposition \ref{pro9} imply 
\[ \overline{\lambda}=\frac{8\pi P(I_d)}{\left( \int_{I_d}\alpha P(d\alpha)\right)^2} \] 
and therefore, 
\begin{equation} 
\overline{\lambda}>\frac{8\pi}{\left( \int_{I_+}\alpha P(d\alpha)\right)^2} 
 \label{47.1}
\end{equation} 
by Proposition \ref{pro5}. The right-hand side of (\ref{47.1}), however, is equal to $\overline{\lambda}$ by Proposition \ref{pro3} and Propoisition \ref{pro4}. 

Second, if $\alpha_d\leq \alpha_{\inf}$, we have 
\begin{equation} 
\overline{\lambda}=\frac{8\pi P(I_{\inf})}{\left( \int_{I_{\inf}}\alpha P(d\alpha)\right)^2} 
 \label{47.2}
\end{equation}  
by the third and the fourth cases of Proposition \ref{pro9}. By (\ref{42.0}), therefore, we obtain 
\[ \overline{\lambda}>\frac{8\pi}{\left( \int_{I_+}\alpha P(d\alpha)\right)^2} \] 
by Proposition \ref{pro5}, a contradiction again. 

It thus holds that $\alpha_+=\alpha_{\inf}$, and therefore, if $\alpha_+\not\in {\cal A}$, we obtain $I_{\inf}=(\alpha_{\inf},1]=(\alpha_+, 1]$. This property implies a contradiction again by (\ref{47.2}) and Proposition \ref{pro5}, from the assumption $P(\{\alpha_+\})>0$. \hfill $\Box$ 

\begin{lemm}\label{lem4.2}
If (\ref{9.2}) holds, there arises that (\ref{19}) and 
\begin{equation} 
\alpha_+>\frac{4\pi}{\beta_0}, \ \beta_0=\frac{8\pi}{\int_{I_+}\alpha P(d\alpha)}. 
 \label{50.1}
\end{equation} 
\end{lemm} 

{\it Proof:} \ We have 
\[ \overline{\lambda}=\frac{8\pi}{\left( \int_{I_+}\alpha P(d\alpha)\right)^2} \] 
by Proposition \ref{pro3} and Proposition \ref{pro4}, and therefore, (\ref{19}) implies the second case of (\ref{50.1}). Then it holds that 
\[ \frac{4\pi}{\beta_0}=\frac{1}{2}\int_{I_+}\alpha P(d\alpha)\leq \frac{1}{2}, \] 
and therefore, the first case is a consequence of $\alpha_+>\frac{1}{2}$. It thus suffices to prove (\ref{19}). 

Assume, first, 
\begin{equation} 
P(\{\alpha_+\})=0. 
 \label{p1}
\end{equation} 
If $\alpha_d>\alpha_+$, we obtain  
\[ \frac{8\pi}{\left( \int_{I_+}\alpha P(d\alpha)\right)^2}=\overline{\lambda}<\frac{8\pi P(I_d)}{\left( \int_{I_d}\alpha P(d\alpha)\right)^2} \] 
by Proposition \ref{pro3}, Proposition \ref{pro4}, and Proposition \ref{pro5}, while the second and the fourth cases of Proposition \ref{pro9} ensure 
\[ 8\pi \int_{I_d}P(d\alpha)=\overline{\lambda}\left( \int_{I_d}\alpha P(d\alpha)\right)^2, \] 
a contradiction. 

It thus holds that $\alpha_d\leq \alpha_+$. Then the first and the second cases of Proposition \ref{pro9} imply 
\begin{equation} 
\beta_0=\overline{\lambda}
\int_{I_+}\alpha \tilde \psi(\alpha)P(d\alpha)
=\overline{\lambda}
\int_{I_d}\alpha P(d\alpha). 
 \label{50.2}
\end{equation} 
Here we confirm  
\begin{equation}  
\int_{I_d}\alpha P(d\alpha)=\int_{I_+}\alpha P(d\alpha), 
 \label{50.3}
\end{equation} 
In fact, inquality (\ref{50.3})  is obvious if $\alpha_d<\alpha_+$ or $\alpha_d=\alpha_+$ and $I_d=(\alpha_d,1]$. Equality (\ref{50.3}) is still valid in the rest case of $\alpha_+=\alpha_d$ and $I_d=[\alpha_d,1]$ from the assumption (\ref{p1}). We thus obtain (\ref{19}) by (\ref{50.2})-(\ref{50.3}) in the case of (\ref{p1}). 

Assume, second, 
\begin{equation} 
P(\{\alpha_+\})>0. 
 \label{p2}
\end{equation}
If $\alpha_d>\alpha_{\inf}$ we obtain a contradiction similarly to the first case of (\ref{p1}), because $\alpha_d>\alpha_+$ arises in this case by $\alpha_{\inf}\geq \alpha_+$. Thus we obtain $\alpha_d\leq \alpha_{\inf}$, which ensures (\ref{47.7}) by the third case of Proposition \ref{pro9}. It holds also that $\alpha_+=\alpha_{\inf}\in {\cal A}$ by Lemma \ref{lem4.1}, which is valid in the case of (\ref{p2}), and therefore, 
\[ I_{\inf}=[\alpha_+,1]. \] 
Then, the first case of Proposition \ref{pro9} ensures 
\[ \beta_0=\overline{\lambda}\int_{I_{\inf}}\alpha P(d\alpha)=\overline{\lambda}\int_{I_+}\alpha P(d\alpha), \] 
that is, (\ref{19}) even in the second case of (\ref{p2}). \hfill $\Box$ 

\vspace{2mm} 

Having (\ref{19}) and the second relation of (\ref{50.1}),  we can show the following fact by the proof of Theorem 1.2 of \cite{st18}. In fact, there, the condition $P(\alpha_+)>0$ is used only in the proof of Proposition 3.1, which is replaced by Lemma \ref{lem4.2} in our case.  

\begin{prop}\label{pro10}
Let $\mbox{supp} \ P\subset I_+$, and assume $s^+=0$ in (\ref{rv}), (\ref{19}), and $\alpha_+>\frac{4\pi}{\beta_0}$. Then it holds that (\ref{9}) with $\beta_k\rightarrow \beta_0$ satisfying 
\begin{equation} 
\beta_k\leq \lambda_k\int_{I_+}\alpha P(d\alpha). 
 \label{p3}
\end{equation} 
\end{prop}  

Inequality (\ref{p3}) ensures $\beta_k\leq \beta_0$ as is desired. 
Now we complete the following proof. 

\vspace{2mm} 

{\it Proof of Theorem \ref{thm2}:} \ We can argue similarly to \cite{st19}, having confirmed (\ref{9})-(\ref{14.2}) 
and (\ref{50.1}).  Here we just draw a sketch of the proof. 

In fact we have 
\begin{eqnarray} 
& & \frac{2}{\lambda_k}J_{\lambda_k}(v_k)=\iint_{I_+\times \Omega}e^{w_k^\alpha}\left(w_k^\alpha-w_k^\alpha(0)\right)dxP(d\alpha) \nonumber\\ 
& & \quad +\int_{I_+}\left(w_k^\alpha(0)-\log \int_\Omega e^{\alpha v_k}\right) P(d\alpha). 
 \label{54.1}
\end{eqnarray} 
For the first term on the right-hand side of (\ref{54.1}) we use (\ref{9}), regarding 
\[ \alpha_+>\frac{4\pi}{\beta_0} \] 
and (\ref{14.2}) 
as in (4.24) of \cite{st19}. For the second term on the right-hand side of (\ref{54.1}), we use (\ref{9}) with  
\[ \beta_0=\frac{8\pi}{\int_{I_+}\alpha P(d\alpha)} \] 
and (\ref{14.2}).  Then (\ref{bounded}) follows. \hfill $\Box$

\section{Two-sided measures: proof of Theorem \ref{thm3}}\label{sec5}

Here we assume 
\begin{equation} 
I_\pm\cap \mbox{supp} \ P\neq\emptyset  
 \label{57.1}
\end{equation} 
and take $(v_k, \lambda_k)$, $k=1,2,\cdots$, as in (\ref{3.7}) and (\ref{111}). There still arises that (\ref{4})-(\ref{6}). 

\begin{lemm}\label{lem5.1}
Any $x_0\in {\cal S}$ admits either $n^+(x_0)\geq 8\pi$ or $n^-(x_0)\geq 8\pi$.  If 
$n^{+}(x_0)\geq 8\pi$ and 
\begin{equation} 
\int_{I_+}\alpha P(d\alpha)>\frac{P(I_+)}{1+\frac{\sqrt{5}}{2}} 
 \label{para3}
\end{equation} 
it holds that either $x_0\not \in {\cal S}_{-}$ or $n^{-}(x_0)\geq 24\pi$. If $n^-(x_0)\geq 8\pi$ and 
\begin{equation} 
\int_{I_-}\vert \alpha\vert P(d\alpha)>\frac{P(I_-)}{1+\frac{\sqrt{5}}{2}}, 
 \label{para4}
\end{equation} 
similarly, it holds that either $x_0\not\in {\cal S}_+$ or $n^+(x_0)\geq 24\pi$. 
\end{lemm} 

To make the statement simple, we say that $n^\pm(x_0)\geq 8\pi$ and 
\begin{equation} 
\int_{I_\pm}\vert \alpha \vert P(d\alpha)>\frac{P(I_\pm)}{1+\frac{\sqrt{5}}{2}} 
 \label{root5}
\end{equation} 
imply either $x_0\not\in {\cal S}_\mp$ or $n^\mp(x_0)\geq 24\pi$ for the second and the third assertions of Lemma \ref{lem5.1}.  Similar descriptions are also used below. 

\vspace{2mm} 

{\it Proof of Lemma \ref{lem5.1}:} \ It holds that  
\begin{eqnarray*}  
m^\pm(x_0) & = &\int_{I_\pm}m(\alpha, x_0)P(d\alpha) \\ 
& \geq & n^\pm(x_0)=\int_{I_\pm}\vert \alpha\vert m(\alpha, x_0)P(d\alpha), \ x_0\in {\cal S}_\pm.  
\end{eqnarray*}
Put  
\begin{eqnarray*} 
& & n^+(x_0)=m^+(x_0)=0, \quad x_0\not \in {\cal S}_+ \\ 
& & n^-(x_0)=m^-(x_0)=0, \quad x_0\not\in {\cal S}_-
\end{eqnarray*} 
for simplicity. 
Then we obtain 
\begin{eqnarray} 
8\pi(n^+(x_0)+n^-(x_0)) & \leq & 8 \pi (m^+(x_0)+m^-(x_0)) \nonumber\\ 
& = & (n^+(x_0)-n_-(x_0))^2 \label{60} 
\end{eqnarray}  
by (\ref{6}). Here the curve 
\[ 8\pi(X_++X_-)=(X_+-X_-)^2 \] 
in $X_-X_+$plane is a parabola passing through $(0,8\pi)$, $(0,0)$, and $(8\pi,0)$, as is drawn in a figure in \cite{os06} p.287 or \cite{rt16} p. 152.  

Since 
\[ (n^-(x_0), n^+(x_0))\neq (0,0), \ \forall x_0\in {\cal S}, \] 
the value $(n^-(x_0), n^+(x_0))$ belongs to one of the two components $Z_\pm$ in $X_-X_+$ plane, defined by  
\[ Z_\pm=\{ (X_-,X_+) \mid 8\pi(X_++X_-)\leq (X_+-X_-)^2, \ X_\pm>X_\mp\} \setminus \{0,0\}. \] 
In particular, we have either $n^+(x_0)\geq 8\pi$ or $n^-(x_0)\geq 8\pi$. 

Assume $n^+(x_0)\geq 8\pi$, and let $(n^-(x_0), n^+(x_0))\in Z_+$. Suppose, furthermore, $x_0\in {\cal S}_-$. Then, it follows that 
\[ n^-(x_0)\geq 4\pi, \] 
and therefore,   
\begin{equation} 
n^+(x_0)\geq (8+4\sqrt{5})\pi. 
 \label{parabola1}
\end{equation} 

For $\mu_k(dxd\alpha)$ define by (\ref{muk}), on the other hand, we obtain 
\begin{eqnarray*} 
& & \iint_{\Omega\times I_+}\alpha\mu_k(dxd\alpha)= \lambda_k\int_{I_+}\alpha P(d\alpha) \\ 
& & \quad \rightarrow \int_{I_+}\alpha\left( \sum_{x_0'\in {\cal S}}m(x_0',\alpha)+\Vert s_1\Vert_1\right)P(d\alpha) \geq n^+(x_0), 
\end{eqnarray*} 
and therefore, 
\begin{equation} 
\overline{\lambda}\int_{I_+}\alpha P(d\alpha)\geq 8\pi\left(1+\frac{\sqrt{5}}{2}\right). 
 \label{para2}
\end{equation} 
Since 
\begin{equation} 
\overline{\lambda}\leq \frac{8\pi P(I_\pm)}{\left(\int_{I_\pm}\alpha P(d\alpha)\right)^2} 
 \label{pm+}
\end{equation} 
inequality (\ref{para2}) is impossible by (\ref{para3}). Thus we obtain $x_0\not\in {\cal S}_-$ in this case. 

In  the other case of $(n^-(x_0), n^+(x_0))\in Z_-$ we obtain $n^-(x_0)\geq 24\pi$ by $n^+(x_0)\geq 8\pi$. Similarly, if $n^-(x_0)\geq 8\pi$, we obtain either $x_0\not\in {\cal S}_+$ or $n^+(x_0)\geq 24\pi$. \hfill $\Box$ 

\begin{lemm}\label{lem5.3}
Under the assumption of (\ref{root5}) if 
\begin{equation} 
n^\pm(x_0)\geq 8\pi, \ x_0\in {\cal S}_\pm, 
 \label{58.1}
\end{equation} 
it holds that $x_0\not\in {\cal S}_\mp$, provided that  
\begin{equation} 
P(I_\mp)<3\int_{I_\mp}\vert \alpha \vert P(d\alpha). 
 \label{58.2}
\end{equation}  
\end{lemm} 

{\it Proof:} \ From (\ref{root5}) and (\ref{58.1}) we obtain either ${\cal S}_\mp=\emptyset$ or $n^{\mp}(x_0)\geq 24\pi$ by Lemma \ref{lem5.1}. In the latter case we use (\ref{4}) and (\ref{mu})-(\ref{muk}), to obtain 
\begin{eqnarray*}    
& & \lambda_k\int_{I_\mp}\vert \alpha \vert P(d\alpha)=\lambda_k\iint_{\Omega\times I_\mp}\vert \alpha  \vert \mu_k(dxd\alpha) \\ 
& & \quad \rightarrow \iint_{I_\mp\times \Omega}\vert \alpha\vert \mu(dxd\alpha)\geq \int_{I_\mp}\vert \alpha \vert m^{\mp}(x_0,\alpha) P(d\alpha)=n^{\mp}(x_0)\geq 24\pi  
\end{eqnarray*} 
or, 
\begin{equation}  
\overline{\lambda}\geq \frac{24\pi}{\int_{I_\mp}\vert \alpha \vert P(d\alpha)}. 
 \label{58.3}
\end{equation} 
Inequality (\ref{58.3}), however, is impossible because (\ref{pm+}) and (\ref{58.2}) imply 
\[ \overline{\lambda}\leq \frac{8\pi P(I_\mp)}{\left( \int_{I_\mp}\vert \alpha \vert P(d\alpha)\right)^2}<\frac{24\pi}{\int_{I_\mp}\vert \alpha\vert P(d\alpha)}. \] 
\hfill $\Box$ 

\begin{lemm}\label{lem5.4}
Assume, both 
\begin{equation} 
\alpha_\pm>\frac{1}{2},   
 \label{pm}
\end{equation} 
and (\ref{111}) in (\ref{3.7}). Then it holds that $\sharp {\cal S}_\pm\leq 1$. Each $x_0\in {\cal S}$ admits either 
\begin{equation} 
x_0\in {\cal S}_+\setminus {\cal S}_-, \ n^+(x_0)\geq 8\pi 
 \label{64.1}
\end{equation} 
or 
\begin{equation} 
x_0\in {\cal S}_-\setminus {\cal S}_+, \ n^-(x_0)\geq 8\pi. 
 \label{64.2}
\end{equation} 
In each case there arises also that $s^\pm=0$, although $s^\mp\not\equiv 0$ can happen, and $\overline{\lambda}$ in (\ref{3}) is attained by $K_\pm=I_\pm$. 
\end{lemm} 

{\it Proof:} \ Inequality (\ref{pm}) excludes the possibility of (\ref{root5}). Also we have (\ref{58.2}) as in 
\begin{equation} 
3\int_{I_\pm}\alpha P(d\alpha)>\frac{3}{2}P(I_\pm)>P(I_\pm). 
 \label{63.1}
\end{equation} 
These properties imply 
\begin{equation} 
{\cal S}_+\cap {\cal S}_-=\emptyset
 \label{65.1}
\end{equation} 
by Lemma \ref{lem5.1} and Lemma \ref{lem5.3}, and therefore, any $x_0\in {\cal S}$ admits (\ref{64.1}) or (\ref{64.2}). 

Here we put 
\begin{equation} 
f_k^\pm=\lambda_k\int_{I_\pm}\vert \alpha\vert \frac{e^{\alpha v_k}}{\int_\Omega e^{\alpha v_k}}P(d\alpha)\geq 0 
 \label{68.1}
\end{equation}  
and define $v_k^\pm=v_k^\pm(x)$ by (\ref{x10}) for $(v,f)=(v_k^\pm, f_k^\pm)$, that is, 
\begin{equation}  
-\Delta v_k^\pm=f_k^\pm-\frac{1}{\vert \Omega\vert}\int_\Omega f_k^\pm, \ \int_\Omega f_k^\pm=0. 
 \label{68.2}
\end{equation} 
It holds that 
\begin{equation} 
v_k^\pm \geq -C, \ v_k=v_k^+-v_k^-.  
 \label{68.3}
\end{equation} 

If $x_0\in {\cal S}_\pm$ we obtain  
\[ f_k^\pm\rightharpoonup n^\pm(x_0)\delta_{x_0} +\tilde s^\pm \ \mbox{in ${\cal D}'(B_R(x_0))$}, \quad n^\pm(x_0)\geq 8\pi \] 
by (\ref{4}) and (\ref{5.5}) for $0<R\ll 1$, which implies 
\[ \lim_{k\rightarrow \infty}\int_\Omega e^{\pm\alpha_{\pm}v_k^\pm}=+\infty \]
by (\ref{pm}) as in the proof of Proposition \ref{pro3}.  If ${\cal S}_\pm\neq \emptyset$, therefore, it follows that 
\[ \mbox{supp} \ \mu(dxd\alpha)\cap (\Omega\times I_\pm)\subset {\cal S}_\pm, \] 
and then, the residual vanishing $s^\pm=0$ occurs in (\ref{rv}).  

If $x_0\in {\cal S}_+$, in particular, we obtain  
\[ n^+(x_0)\geq 8\pi, \ x_0\not\in {\cal S}_-, \ s^+=0.  \] 
Then the convergence (\ref{4}) implies 
\begin{eqnarray*} 
& & \iint_{I_+\times\Omega}\mu_k \ \rightarrow \ \sum_{x_0'\in {\cal S}_+}\int_{I_+}m(\alpha, x_0')P(d\alpha) \\ 
& & \iint_{I_+\times\Omega}\alpha \mu_k \ \rightarrow \ \sum_{x_0'\in {\cal S}_+}\int_{I_+}\alpha m(\alpha, x_0')P(d\alpha) 
\end{eqnarray*} 
for $\mu_k(dxd\alpha)$ defined by (\ref{muk}). Since 
\[ \iint_{I_+\times\Omega}\mu_k=\lambda_kP(I_+), \  \iint_{I_+\times\Omega}\alpha \mu_k=\lambda_k\int_{I_+}\alpha P(d\alpha) \] 
it follows that  
\begin{eqnarray} 
& & \overline{\lambda}P(I_+)=\sum_{x_0'\in {\cal S}_+}m(\alpha, x_0')P(d\alpha)  \label{67.1} \\ 
& & \overline{\lambda}\int_{I_+}\alpha P(d\alpha)=\sum_{x_0'\in {\cal S}_+}\int_{I_+}\alpha m(\alpha, x_0')P(d\alpha). 
 \label{67.2}
\end{eqnarray} 

We have, on the other hand,  
\[ 8\pi\int_{I_+}m(\alpha, x_0')=\left( \int_{I_+}\alpha m(\alpha, x_0')P(d\alpha)\right)^2, \ x_0'\in {\cal S}_+ \] 
by (\ref{6}), because  (\ref{65.1}) and $s^+=0$ hold true. It thus follows that 
\begin{eqnarray*} 
8\pi\overline{\lambda}P(I_+) & = & \sum_{x_0'\in {\cal S}_+}8\pi\int_{I_+}m(\alpha, x_0')P(d\alpha) \\ 
& = & \sum_{x_0'\in {\cal S}_+}\left( \int_{I_+}\alpha m(\alpha, x_0')P(d\alpha\right)^2 \\ 
& \leq & \left( \sum_{x_0'\in {\cal S}_+}\int_{I_+}\alpha m(\alpha, x_0')P(d\alpha)\right)^2 \\ 
& = & \overline{\lambda}^2\left( \int_{I_+}\alpha P(d\alpha)\right)^2
\end{eqnarray*}  
from (\ref{67.1})-(\ref{67.2}), and therefore,  
\[ \frac{8\pi P(I_+)}{\left( \int_{I_+}\alpha P(d\alpha)\right)^2}\leq \overline{\lambda}. \] 
This property means that $\overline{\lambda}$ is attained by $K_+=I_+$, and there arises aslo that ${\cal S}_+=\{x_0\}$. 

The other case of ${\cal S}_-\neq\emptyset$ is treated similarly. \hfill $\Box$ 

\vspace{2mm} 

We continue to assume (\ref{57.1}) and (\ref{pm}). Recall $(v_k,\lambda_k)$, $k=1,2,\cdots$, with (\ref{3.7}) and (\ref{111}). 

\vspace{2mm}

{\it Proof of Theorem \ref{thm3}:} \ We have either  
$(\sharp {\cal S}_+, \sharp{\cal S}_-)=(1,0)$, ${(\sharp \cal S}_+, \sharp {\cal S}_-)= (0,1)$, or $(\sharp{\cal S}_+, \sharp{\cal S}_-)=(1,1)$ by Lemma \ref{lem5.4}. Here we take the last case because the other cases are easier. 

We thus assume $\sharp {\cal S}_\pm=1$ and put ${\cal S}_\pm=\{ x_\pm\}$. By Lemma \ref{lem5.4} we have 
\[ f_k^\pm\rightharpoonup n^\pm(x_\pm)\delta_{x_\pm} \] 
for $f_k^\pm=f_k^\pm(x)$ defined by (\ref{68.1}), and therefore, 
\begin{equation} 
-\Delta v_k^\pm \rightharpoonup \mu^\pm\equiv n^{\pm}(x_\pm)\delta_{x_\pm}-\frac{n^\pm(x_\pm)}{\vert \Omega\vert} \quad \mbox{in ${\cal M}(\Omega)$} 
 \label{101}
\end{equation} 
and 
\begin{eqnarray} 
& & v_k^\pm\rightarrow v^\pm\equiv n^\pm G(\cdot, x_\pm) \nonumber\\ 
& & \mbox{in $C^\infty$ topology on each compact set in $\Omega\setminus \{ x_\pm\}$} 
 \label{102}
\end{eqnarray} 
by the elliptic regularity. 

Since 
\[ J_{\lambda_k}\geq J_{\overline{\lambda}} \] 
holds by 
\[ J_\lambda(v)=\frac{1}{2}\Vert \nabla v\Vert_2-\lambda\int_\Omega \left( \log \frac{1}{\vert \Omega\vert}\int_\Omega e^{\alpha v}\right)P(d\alpha)-\lambda\log \vert\Omega\vert \] 
it suffices to derive  
\begin{equation} 
J_{\overline{\lambda}}(v_k)\geq -C 
 \label{104}
\end{equation} 
to ensure (\ref{bounded}).  

In fact, using $v_k^\pm=v_k^\pm(x)$ defined by (\ref{68.2}), we obtain 
\begin{eqnarray*} 
J_{\overline{\lambda}}(v_k) 
& = & \frac{1}{2}
\Vert \nabla v_k\Vert_2^2-\overline{\lambda}\int_I\left( \log \int_\Omega e^{\alpha v_k}\right) P(d\alpha) 
\\ 
& = & \frac{1}{2}\Vert \nabla v_k\Vert_2^2-\overline{\lambda}\int_{I_+}\left( \log \int_\Omega e^{\alpha v_k}\right)P(d\alpha) \\ 
& & -\overline{\lambda} \int_{I_-}\left( \log \int_\Omega e^{\alpha v_k}\right)P(d\alpha) \\ 
& \geq & \frac{1}{2}\Vert \nabla v_k\Vert_2^2-\overline{\lambda}\int_{I_+}\left( \log \int_\Omega e^{\alpha v_k^++C}\right)P(d\alpha) \\ 
& & - \overline{\lambda}\int_{I_-}\log \left( \int_\Omega e^{-\alpha v_k^-+C}\right) P(d\alpha) 
\end{eqnarray*} 
by (\ref{68.3}), and therefore, 
\begin{eqnarray*} 
J_{\overline{\lambda}}(v_k)
& \geq  & \frac{1}{2}\Vert \nabla v_k\Vert_2^2
-\overline{\lambda}\int_{I_+}\left( \log \int_\Omega e^{\alpha v_k}\right)P(d\alpha) \\ 
& & -\overline{\lambda}\int_{I_-}\left(\log \int_\Omega e^{\alpha v_k}\right)P(d\alpha)-C. 
\end{eqnarray*}

Here we put 
\[ P=\tau_+P_++\tau_-P_-, \ \tau_\pm>0, \ \tau_1+\tau_-=1, \ \mbox{supp} \ P_\pm \subset I_\pm, \] 
recalling (\ref{57.1}), where $P_\pm=P_\pm(d\alpha)$ are probability measues on $I_\pm$. Then it holds that
 
\begin{eqnarray}\label{100} 
J_{\overline{\lambda}}(v_k) & \geq & 
\left\{ \frac{1}{2}\Vert \nabla v_k^+\Vert_2^2-\tau_+\overline{\lambda}
\int_{I_+}\left( \log \int_\Omega e^{\alpha v_k^+}\right)P_+(d\alpha)
\right\} \nonumber\\ 
& & +\left\{ \frac{1}{2}\Vert \nabla v_k^-\Vert_2^2-\tau_-\overline{\lambda}\int_{I_-}\left( \log \int_\Omega e^{-\alpha v_k^-}\right)P_-(d\alpha)\right\} \nonumber\\ 
& & -\int_\Omega \nabla v_k^+\cdot \nabla v_k^- -C. 
\end{eqnarray} 
We have 
\begin{eqnarray*} 
\overline{\lambda} & \leq & \inf \left\{ \frac{8\pi P(K_+)}{\left( \int_{K_+}\alpha P(K_+)\right)^2} \mid K_+\subset I_+\cap \mbox{supp} \ P, \ \mbox{Borel set} \right\} \\ 
& = &  \frac{1}{\tau_+}
 \inf \left\{ \frac{8\pi P_+(K_+)}{\left( \int_{K_+}\alpha P_+(K_+)\right)^2} \mid K_+\subset I_+\cap \mbox{supp} \ P_+, \ \mbox{Borel set} \right\} 
\end{eqnarray*} 
and therefore, 
\[ \overline{\lambda}_+\equiv \inf \left\{ \frac{8\pi P_+(K_+)}{\left( \int_{K_+}\alpha P_+(d\alpha)\right)^2} \mid K_+\subset I_+\cap \mbox{supp} \ P_+, \ \mbox{Borel set}
\right\} \geq \tau_+\overline{\lambda}. \] 
Similarly, we have 
\[ \overline{\lambda}_-
\equiv \inf 
\left\{ \frac{8\pi P_-(K_-)}{\left( \int_{K_-}\alpha P_-(d\alpha)\right)^2} 
\mid K_-\subset I_-\cap \mbox{supp} \ P_-, \ \mbox{Borel set}
\right\} \geq \tau_-\overline{\lambda}, 
\] 
and therefore,  
\begin{equation} 
J_{\overline{\lambda}}(v_k)\geq -\int_\Omega \nabla v_k^+\cdot \nabla v_k^--C 
 \label{103}
\end{equation} 
by Theorem \ref{thm2} because  
\[ \min \mbox{supp} \ P_+=\alpha_+>\frac{1}{2}, \ \max \mbox{supp} \ P_-=-\alpha_-<-\frac{1}{2} \] 
holds true.

We take $0\leq \rho_\pm \in C^\infty(\Omega)$ such that 
\[ \rho_++\rho_-=1, \quad \rho_\pm=\left\{ \begin{array}{ll} 
1, & \mbox{near $x_\pm$} \\ 
0, & \mbox{near $x_\mp$}. \end{array} \right.  \] 
Then, we have    
\[ \int_\Omega \nabla v_k^+\cdot\nabla v_k^-=\int_\Omega \rho_+\nabla v_k^+\cdot \nabla v_k^-+\int_\Omega \rho_-\nabla v_k^+\cdot \nabla v_k^- \] 
with 
\begin{eqnarray*} 
\int_\Omega \rho_+\nabla v_k^+\cdot \nabla v_k^- & = & -\int_\Omega [\nabla\cdot (\rho_+\nabla v_k^+)]v_k^- \\ 
& = & -\int_\Omega (\nabla\rho_+\cdot \nabla v_k^+)v_k^-+\int_\Omega (-\Delta v_k^+)\rho_+v_k^-. 
\end{eqnarray*} 
Since $\rho_+=0$ near $x_-$ and $\nabla \rho_+=0$ near $x_\pm$ both, it follows that 
\[ \int_\Omega\rho_+\nabla v_k^+\cdot \nabla v_k^-\rightarrow -\int_\Omega (\nabla \rho_+\cdot \nabla v^+)v^-+\langle \rho_+v^-, \mu^+\rangle \] 
from (\ref{101})-(\ref{102}). 

Similarly, we obtain 
\[ \int_\Omega\rho_-\nabla v_k^+\cdot \nabla v_k^-\rightarrow -\int_\Omega (\nabla \rho_-\cdot \nabla v^-)v^++\langle \rho_-v^+, \mu^-\rangle \] 
and  hence (\ref{104}) by (\ref{103}). \hfill $\Box$

\section*{Acknowledgement}
The authors thank Prof. H. Ohtsuka and Prof. R. Takahashi for valuable comments on our study. The first author is supported by JSPS Grand-in-Aid for Scientific Reserach 25K07067.

{}
\end{document}